\documentclass[11pt, a4paper]{article} 

\usepackage[dvipsnames]{xcolor}
\usepackage{amsfonts}
\usepackage{amssymb}
\usepackage{amsmath}
\usepackage{amsthm}
\usepackage{latexsym}
\usepackage{graphicx}
\usepackage{epstopdf}
\usepackage[retainorgcmds]{IEEEtrantools}
\usepackage{pdfpages}
\usepackage{enumerate}
\usepackage{caption}
\usepackage{subcaption}	
\usepackage{amsthm}
\usepackage{setspace}
\usepackage{dsfont}
\usepackage{bbm}
\usepackage{hyperref}
\usepackage{datetime}
\usepackage[colorinlistoftodos, textsize=footnotesize, backgroundcolor=GreenYellow!25, linecolor=GreenYellow!25, bordercolor=GreenYellow!60]{todonotes}
\usepackage[percent]{overpic}
\usepackage{pict2e}
\presetkeys{todonotes}{fancyline}{}

\makeatletter

\newdimen\mainfontsize \mainfontsize=1\@ptsize pt
\topskip=\mainfontsize
   \@maxdepth=\maxdepth

\makeatother

\makeatletter
\renewcommand\@makefntext[1]{%
  \noindent\makebox[0.5em][r]{\@makefnmark}#1}
\makeatother

\makeatletter
\def\blfootnote{\xdef\@thefnmark{}\@footnotetext}
\makeatother

\newtheorem{theorem}{Theorem}[section]
\newtheorem{corollary}{Corollary}[section]

\theoremstyle{remark}

\theoremstyle{definition}

\newtheorem{remark}{Remark}[section]

\newcommand{\E}{\mathbb E}
\newcommand{\R}{\mathbb R}

\newcommand{\F}{\mathcal F}

\newcommand{\ud}{\,\mathrm{d}}
\newcommand{\vd}{\mathrm{d}}
\newcommand{\lt}{\left}
\newcommand{\rt}{\right}
\newcommand{\pt}{\partial}

\def\P{{\mathbb P}}

\newcommand{\Ind}{\mathbbm{1}}

\newcommand{\supp}{\mathop{\mathrm{supp}}\nolimits}

\newcommand{\Pdom}{D}
\DeclareMathOperator*{\argmin}{arg\,min}
\DeclareMathOperator*{\argmax}{arg\,max}

\makeatletter
\@addtoreset{equation}{section}

\makeatother

\title{Monotonicity and robustness in Wiener disorder detection}       
\author{Erik Ekstr\"om\footnote{Department of Mathematics, Uppsala University, Box 480, 751 06 Uppsala, Sweden 
\mbox{(email: \href{mailto:ekstrom@math.uu.se}{\nolinkurl{ekstrom@math.uu.se}}.})} \and 
Juozas Vaicenavicius \footnote{Department of Information Technology, Uppsala University, Box 337, 751 05 Uppsala, Sweden \mbox{(email: \href{mailto: juozas.vaicenavicius@it.uu.se}{\nolinkurl{juozas.vaicenavicius@it.uu.se}})}}
 }

\date{ }   

\begin{document}

\maketitle
\begin{abstract}
We study the problem of detecting a drift change of a Brownian motion under various extensions of 
the classical case. Specifically, we consider the case of a random post-change drift and examine monotonicity properties of the solution with respect to different model parameters.
Moreover, robustness properties -- effects of misspecification of the underlying model -- are explored.
\end{abstract}



\section{Introduction} \label{s:Intro}

In the classical version of the quickest disorder detection (QDD) problem, see \cite{aS67}, one observes a one-dimensional process $Y$ which satisfies
\[Y_t=b(t-\Theta)^++\sigma W_t,\]
where $b$ and $\sigma$ are non-zero constants, $W$ is a standard Brownian motion and the disorder time $\Theta$
is an exponentially distributed random variable (with intensity $\lambda>0$) such that $W$ and $\Theta$ are independent. 
The associated Bayes' risk (expected cost) corresponding to a stopping rule $\tau$ is defined as
\begin{eqnarray}\label{E:Bayes}
\P(\Theta>\tau)+c\E[(\tau-\Theta)^+],
\end{eqnarray}
where $c>0$ is the cost of one unit of detection delay. It is well-known (see \cite[Chapter 4]{S78}) that to minimise the Bayes risk
one should stop the first time the conditional probability process $\Pi_t:=\P(\Theta\leq t\vert \mathcal F_t^Y)$ reaches a certain level $a$. 
Moreover, the level $a$ is characterized as the unique solution of a transcendental equation.

In many situations, however, it is natural {\em not} to know the exact value of the disorder magnitude $b$, 
but merely its {\em distribution}. This is the case for example when a specific machine is monitored continuously, and the
machine can break down in several possible ways.
To study such a situation, we allow for the new drift to be a random variable $B$ with distribution $\mu$ such that $B$ is independent 
of the other sources of randomness.
In this setting we study {\bf monotonicity properties} 
of the QDD problem, i.e. whether the (minimal) expected cost 
is monotone with respect to various model parameters. In particular, we study the dependence of the expected cost
on the volatility $\sigma$, the distribution $\mu$, and
the disorder intensity $\lambda$. 
We also study {\bf robustness} in the QDD problem, i.e.
what happens if one misspecifies various model parameters. More specifically, 
we aim at estimates for the increased cost associated with the use of suboptimal strategies. 
Clearly, such estimates are helpful in situations where the model is badly calibrated, but also in situations where one chooses to use a simpler suboptimal strategy rather than a computationally more demanding optimal strategy. 

As mentioned above, the classical version of the QDD problem was studied in \cite{aS67}, see also \cite[Chapter 4]{S78}
and \cite[Section 22]{PS06}; for extensions to the case of detecting a change in the intensity of a Poission process, 
see \cite{PS02}, \cite{BDK} and \cite{BDK06}.
For the case of a random disorder magnitude, \cite{B} obtains asymptotic results of a problem with normally distributed drift.
Concavity of the 
value function in a related hypothesis testing problem with two possible post-change drift values in a time-homogeneous case was obtained in \cite{MS}. Finally, practical significance of the disorder detection problem in modern 
engineering applications is explained in \cite{ZTP}.

\section{General model formulation}

\label{S:2}

We model a signal-processing activity on a stochastic basis $(\Omega, \F, \mathbb{F}, \P)$, where the filtration $\mathbb{F} = \{\F_{t}\}_{t\geq0 }$ satisfies the usual conditions.
We are interested in the signal process $X$, which is not directly observable, but we can continuously observe the noisy process
\begin{eqnarray}
Y_{t}= \int_{0}^{t} X_{u} \ud u +  \int_0^t\sigma(u) dW_{u}, \qquad t \geq 0.
\end{eqnarray}
Here $W$ is a Brownian motion independent of $X$, the dispersion $\sigma$ is deterministic and strictly positive, and the signal process follows
\begin{eqnarray}
X_{t}= B^0 \Ind_{\{\Theta =0 \}} +B^1 \Ind_{\{ 0 < \Theta \leq t \}},
\end{eqnarray}
where $\Theta$ is a $[0, \infty)$-valued random variable representing the disorder occurrence time. Moreover, $B^0, B^1$ are real-valued random variables corresponding to disorder magnitudes in the cases `disorder occurs before we start observing $Y$' and `disorder occurs while we observe $Y$', respectively. Also, $\Theta$, $B^0$, and $B^1$ are independent.   
Let $\Theta$ have the distribution $\tilde{\pi}\delta_0 + (1-\tilde{\pi})\nu$, were $\nu$ is a probability measure on $(0,\infty)$ with  a continuously differentiable distribution function $F_\nu$. In addition, denote the distributions of $B^0$ and $B^1$  by  $\mu^0$ and $\mu^1$, respectively. When referring to $\mu^0$ and $\mu^1$ collectively, we will simply say that the prior is $\mu$.  Let us introduce the notation
\[ \Pdom^{n} := \{ \pi \in [0, \infty)^{n}\,:\, \|\pi\|_{1} \leq 1\}\]
and
\[\Delta^n := \left\{\pi \in [0, \infty)^{n} \,:\, \|\pi\|_{1} = 1\right\},\]
where $\|\pi\|_{1}= \sum_{i=1}^n \pi_i$. We assume that
\begin{eqnarray*}
 \mu^0 = \sum_{i=1}^{n} \check{p}_{i} \delta_{b_{i}}, \quad \mu^1 = \sum_{i=1}^{n} p_{i} \delta_{b_{i}},
\end{eqnarray*}
where $b_{1}, \ldots, b_{n}\in \R\setminus\{0\}$ and $(\check{p}_{1}, \ldots, \check{p}_{n}),(p_{1}, \ldots, p_{n}) \in \Delta^n$.

The model studied in the paper is a generalisation of the classical disorder occurrence model \cite{aS67}. Firstly, the exponential disorder distribution used in the classical problem is replaced by an arbitrary distribution with time-dependent intensity. The generalisation is advantageous in situations when the intensity of the disorder occurrence changes with time. For example, if the disorder corresponds to a component failure in a system, for many physical systems, the failure intensity is known to increases with age. Also, if occurrence of the disorder depends on external factors such as weather, then such dependency can be incorporated into the time-dependent disorder intensity from an accurate weather forecast. Moreover, in contrast to the classical problem in which the disorder magnitude is known in advance, in this generalisation, the magnitude takes a value from a range of possible values. Returning to the component failure example, the different possible disorder magnitudes would represent different types of component failure. In the problem of detecting malfunctioning atomic clocks \cite{ZTP}, the disorder corresponds to a systematic drift of a clock. The sign of the disorder magnitude reflects whether a clock starts to go too slow or two fast while the absolute value represents the severity of the drift. In addition, the different distributions $\mu^0$, $\mu^1$ of $B^0$ and $B^1$ and the weight $\tilde \pi$ reflect the prior knowledge about how likely different disorder magnitudes are if the disorder happened before or while observing $Y$. For instance, such model flexibility is relevant when we start observing the system after a particular incident (e.g.~a storm if the system is affected by the weather) and we know that the distribution of possible disorder magnitudes after the incident is different than under normal operating conditions. From a mathematical point of view, $\tilde \pi$ and $B^0$ allows us to give a statistical interpretation to an arbitrary starting point in the Markovian embedding \eqref{E:v} of the original optimal stopping problem studied later. 

\begin{remark}
We point out that the finite support assumption on $\mu$ is made for notational convenience. As any distribution can be approximated arbitrarily well by finitely supported ones, obviously, our monotonicity results below can be extended to general disorder magnitude distributions.  
\end{remark}

We are interested in a disorder detection strategy $\tau$ incorporating two objectives: short detection delay and a small portion of false alarms. As noted in the introduction, a classical choice of Bayes' risk for a detection strategy to minimize is given by
\eqref{E:Bayes}. In the present paper, we consider a slightly more flexible risk structure by allowing a time-dependent cost
for the detection delay. More precisely, we consider the Bayes' risk
\begin{eqnarray*}
R(\tau) &:=& \E\left[ \Ind_{ \{\tau < \Theta \}} + \int_\Theta^\tau c(u) \ud u\right]
\end{eqnarray*} 
where $\Ind_{ \{\tau < \Theta \}}$ is a fixed penalty for a false alarm and the term $\int_\Theta^\tau c(u) \ud u$ is a penalty for detection delay. Here $t \mapsto c(t)$ is a deterministic function with $c(t)>0$ for all $t\geq 0$. Writing $\F^Y=\{\F^Y_t\}_{t\geq0}$ for the filtration 
generated by $Y$ (which is our observation filtration), let us introduce  $\tilde{\Pi}_{t} := \E[\Ind_{\R \setminus \{0\}}(X_{t})\,|\,\F_{t}^{Y} ]$. Then 
\begin{eqnarray}
R(\tau) &=&  \E \lt[1- \E[\Ind_{\{ \Theta \leq \tau \}} \, | \, \F_{\tau}^{Y}]  \rt] + \int_{0}^{\infty}  c(t)\E \lt[  \Ind_{\{ t \leq \tau \}} \E \lt[ \Ind_{\{\Theta \leq t \}} \, | \, \F^{Y}_{t} \rt] \rt] \ud t \nonumber \\
&=& \E \lt[ 1 - \tilde{\Pi}_{\tau} +  \int_{0}^{\tau}c(t) \tilde{\Pi}_{t} \ud t \rt]. \nonumber
\end{eqnarray}
Hence the optimal stopping problem to solve is  
\begin{eqnarray} \label{E:OOS}
V &=& \inf_{\tau \in \mathcal{T}^{Y}} \E \lt[ 1 - \tilde{\Pi}_{\tau} +  \int_{0}^{\tau}c(t) \tilde{\Pi}_{t} \ud t \rt]\,,
\end{eqnarray}
where $\mathcal{T}^{Y}$ denotes the set of $\F^{Y}$-stopping times. 

\subsection{Filtering equations}

Let us define $\Pi^{(i)}_t := \E[\Ind_{\{X_t=b_i\}}\,|\, \F^Y_t]$, where $i=1,\ldots,n$.
By the Kallianpur-Striebel formula, see \cite[Theorem 2.9 on p.~39]{CR11},
\begin{IEEEeqnarray*}{rCl} 
\IEEEeqnarraymulticol{3}{l}{
\Pi^{(i)}_{t}=} \\ 
 && 
 \frac{\tilde{\pi} \check{p}_i e^{\int_{0}^{t} \frac{b_{i}}{\sigma(u)^{2}} \ud Y_{u} - \int_{0}^{t} \frac{ b_{i}^{2}}{2 \sigma(u)^{2}} \ud u}+ 
 (1-\tilde{\pi})p_{i}\int_{[0,t]}  e^{ \int_{\theta}^{t} \frac{b_{i}}{\sigma(u)^{2}} \ud Y_{u} - \int_{\theta}^{t} \frac{ b_{i}^{2}}{2 \sigma(u)^{2}} \ud u} \nu(\vd \theta)}{\tilde{\pi} \sum \limits_{j} \check{p}_j e^{ \int_{0}^{t} \frac{b_{j}}{\sigma(u)^{2}} \ud Y_{u} - \int_{0}^{t} \frac{ b_{j}^{2}}{2 \sigma(u)^{2}} \ud u} + (1-\tilde{\pi})\bigg( \sum \limits_{j} p_j \int_{[0,t]}  e^{ \int_{\theta}^{t} \frac{b_j}{\sigma(u)^{2}} \ud Y_{u} - \int_{\theta}^{t} \frac{ b_j^{2}}{2 \sigma(u)^{2}} \ud u }  \nu(\vd \theta) +  \nu((t, \infty))\bigg) }  \\
 \IEEEyesnumber \label{E:KalStr}
\end{IEEEeqnarray*}
for $i=1,\ldots, n$.
Moreover, from the Kushner-Stratonovich equation, see \cite[Theorem 3.1 on p.~58]{CR11}, we know that $\Pi^{(i)}$ satisfies
\begin{eqnarray} \label{E:Pi}
\vd \Pi^{(i)}_{t} = p_i\lambda({t}) \big(1-\sum_{j=1}^n \Pi^{(j)}_t \big) \ud t + \frac{ \Pi^{(i)}_t}{\sigma(t)}(b_i-\sum_{j=1}^n b_j \Pi^{(j)}_{t})\ud \hat W_{t}, \quad i=1,\ldots,n.
\end{eqnarray}
\noindent Here  $\lambda(t) =F'_{\nu}(t)/(1-F_{\nu}(t))$ is the intensity of the disorder occurring at time $t>0$ (conditional on not having occurred yet), and
\[
 \hat W_{t} =\int_0^t\frac{1}{\sigma(u)}(\ud  Y_{u} - \E[X_{u}\,|\, \F^Y_u]  \ud u ) 
\]
is a standard Brownian motion with respect to $\{ \F^{Y}_{t} \}_{t\geq0}$, see \cite{BC09}
(the process $\hat W_{t}$ is referred to as {\em the innovation process}).
Note that $\tilde \Pi_t=\sum_{i=1}^n\Pi_t^{(i)}$ yields
\begin{eqnarray}
\ud \tilde\Pi_t=\lambda(t)(1-\tilde\Pi_t)\ud t + \frac{\hat X_t}{\sigma(t)}(1-\tilde\Pi_t)\ud \hat W_t,
\end{eqnarray}
where $\hat X_t=\E[X_t\vert\mathcal F^Y_t]$.

The posterior distribution $\P(X_{t} \in \cdot \,|\, \F^{Y}_{t})=\sum_{i=1}^n \Pi^{(i)}_t\delta_{b_i}(\cdot)$, so the $n$-tuple $\Pi_{t} = (\Pi^{(1)}_{t}, \ldots, \Pi^{(n)}_{t})$ fully describes the posterior. As a result, \eqref{E:KalStr} and \eqref{E:Pi} provide two different representations of the posterior distribution. 

\subsection{Markovian embedding}

Following standard lines in optimal stopping theory, we embed our optimal stopping problem into a Markovian framework. 
To do that, define a Markovian value function $V$ by
\begin{eqnarray} \label{E:v}
V(t,\pi) := \inf_{\tau \in \mathcal{T}_{t}^{\Pi}} \E^{t,\pi} \lt[ 1 - \tilde{\Pi}_{t+\tau} + \int_{t}^{t+\tau} c(u)  \tilde{\Pi}_{u} \ud u \rt], \quad (t,\pi)\in [0, \infty)\times D^{n},
\end{eqnarray}
where $\mathcal{T}_t^\Pi$ denotes the stopping times with respect to the $n$-dimensional process $\{\Pi^{t,\pi}_{t+s}\}_{s\geq 0}$ starting from $\pi$ at time $t$ and satisfying \eqref{E:Pi}. It is worth noting that $V(t,\pi)$ corresponds to the value of the problem $\eqref{E:OOS}$ in which the initial time is $t$ and $\mu_0=\sum_{i=1}^n \pi_i\delta_{b_i}$.

\begin{remark}
\label{T:concavity}
The value function $V(t,\cdot)$ in \eqref{E:v} is concave for any $t \geq 0$. Indeed, the concavity proof in \cite{MS} extends to 
the current setting; we omit the details.
\end{remark}

\subsubsection{The classical Shiryaev solution}

In this subsection we recall the solution in the classical case where the cost $c$, the intensity $\lambda$ and the post-change drift $b$ 
are constants. In that case, we have the optimal stopping problem 
\begin{eqnarray}
\label{U}
U(\pi) =\sup_{\tau\in \mathcal{T}^{\Pi}} \E^\pi \lt[ 1 - \Pi_{\tau} +  c \int_{0}^{\tau} \Pi_{t} \ud t \rt]
\end{eqnarray}
with an underlying diffusion process 
\begin{eqnarray*}
\ud\Pi_t=\lambda(1-\Pi_t)\ud t + \frac{b}{\sigma}\Pi_t(1-\Pi_t)\ud \hat W_t.
\end{eqnarray*}
It is well-known (see \cite[Chapter 4]{S78} or \cite[Section 22]{PS06}) that
$U$ solves the free-boundary problem
\begin{eqnarray}
\label{E:vareqn}
\left\{\begin{array}{ll}
\frac{b^2\pi^2(1-\pi)^2}{2\sigma^2}\partial_\pi^2 U + \lambda (1-\pi) \partial_\pi U+ c\pi=0 & \pi\in(0,a)\\
U(\pi)=1-\pi & \pi\in[a,1]\\
\partial_\pi U(a)=-1.\end{array}\right.
\end{eqnarray}
Here $a\in(0,1)$ is the free-boundary, and it can be determined as the solution of a certain transcendental equation.
Moreover, the stopping time $\tau:=\inf\{t\geq 0:\Pi_t\geq a\}$ is optimal in \eqref{U}, 
and one can check that the value function $U$ is decreasing and concave.

\section{Value dependencies and robustness}

\subsection{Monotonicity properties of the value function}

In this section, we study parameter dependence of the optimal stopping problem \eqref{E:v}. In particular, 
we investigate how the value function changes when we alter parameters of the probabilistic model, which include the prior for the drift magnitude and the prior for the disorder time. 

The effects of adding more noise, stretching out the prior by scaling, and increasing the observation cost are explained by the following theorem. 
\begin{theorem}[General monotonicity properties of the value function $V$]\label{T:mon}
\item
\begin{enumerate}
\item  $V$ is increasing in the volatility $\sigma(\cdot)$. \label{first}
\item\label{second}
Given a prior $\mu$ for the drift magnitude, let $V_k$ denote the Markovian value function \eqref{E:v} in the case when the drift prior is
$\mu(\frac{\cdot}{k} )$. 
Then the map $k\mapsto V_k(t,\pi)$ is decreasing on $(0,\infty)$ for any $(t,\pi)$.
\item $V$ is increasing in the cost function $c(\cdot)$.
\end{enumerate}
\end{theorem}

\begin{proof}
For simplicity of notation, and without loss of generality, we consider the case $t=0$ in the proofs below.
\begin{enumerate}
\item
For the volatility, let $t\mapsto \sigma_1(t)$ and $t \mapsto \sigma_2(t)$ be 
two time-dependent 
volatility functions satisying $\sigma_1(t)\leq \sigma_2(t)$ for all $t\geq 0$. Also, let
\[Y^i_t:=\int_{0}^{t} X_{u} \ud u +  \int_0^t\sigma_i(u) dW_{u}, \quad i=1,2,\]
and let $V_i$, $i=1,2$, be the corresponding value functions. 
In addition, let $W^{\perp}$ be a standard Brownian motion independent of $W$ and $X$.
Then, clearly,
\begin{eqnarray*}
V_1=\inf_{\tau\in\mathcal T^{Y^1}}\E\left[ \Ind_{ \{\tau < \Theta \}} + \int_\Theta^\tau c(u)\ud u\right]=
\inf_{\tau\in\mathcal T^{Y^1,W^\perp}}\E\left[ \Ind_{ \{\tau < \Theta \}} + \int_\Theta^\tau c(u)\ud u\right].
\end{eqnarray*}
Moreover, the process 
\[\tilde Y^2_t:=Y^1_t+  \int_0^t\sqrt{\sigma^2_2(u)-\sigma_1^2(u)} dW^\perp_{u}\]
coincides in law with $Y^2$ and $\mathcal T^{\tilde Y^2}\subseteq\mathcal T^{Y^1,W^\perp}$.
Hence it follows that 
\[V_1=\inf_{\tau\in\mathcal T^{Y^1,W^\perp}}\E\left[ \Ind_{ \{\tau < \Theta \}} + \int_\Theta^\tau c(u)\ud u\right]
\leq \inf_{\tau\in\mathcal T^{\tilde Y^2}}\E\left[ \Ind_{ \{\tau < \Theta \}} +\int_\Theta^\tau c(u)\ud u\right]=V_2,\]
which finishes the proof of the claim.
\item
Note that for $k> 0$, the process 
\[ Y^k_t:= \int_{0}^{t} kX_{u} \ud u +  \int_0^t\sigma(u) dW_{u}\]
satisfies $Y_t^k= k\tilde Y_t$, where
\[\tilde{Y}_t:= \int_{0}^{t} X_{u} \ud u +  \int_0^t\frac{\sigma(u)}{k} dW_{u}.\]
Moreover, the set of $\mathcal F^{Y^k}$-stopping times coincides with the set of $\mathcal F^{\tilde Y}$-stopping times,
so monotonicity in $k$ is implied by monotonicity in the volatility. Thus claim \ref{second} follows from claim \ref{first}.
\item
The fact that the value is increasing in $c$ is obvious from the definition \eqref{E:v} of the value function.
\end{enumerate}
\end{proof}

The monotonicity of the minimal Bayes' risk with respect to volatility $\sigma$ is of course not so surprising: more noise in the observation process gives a smaller signal-to-noise ratio, which slows down the speed of learning. It is less 
clear how a change in the disorder intensity $\lambda$ should affect the value function under a general disorder magnitude distribution. However, we have the following comparison result for the case of constant parameters. 

\begin{theorem}[Monotonicity in the intensity for constant parameters]\label{T:monotone}
Assume that the disorder magnitude can only take one value $b\in \R\setminus\{0\}$. Let the cost $c$, the volatility $\sigma$ and the intensity $\lambda$ be constants, and assume that
$\lambda\geq\lambda^\prime(\cdot)$. Let $U$ be the value function for Shiryaev's problem with parameters $(b, \sigma, \lambda, c)$, 
and let $V$ denote the value function for the problem specification $(b, \sigma, \lambda^\prime, c)$.
Then  $U(\pi)\leq V(t,\pi)$ for all $\pi\in[0,1]$ and $t\geq 0$.
\end{theorem}
\begin{proof} 
Without loss of generality, we only consider the case $t=0$.
Let $\pi\in[0,1]$, denote by $Y'$ the observation process corresponding to the model specification $(b, \sigma, \lambda^\prime, c)$,
and let 
$\Pi^\prime$ denote the corresponding process $\Pi$ started from $\pi$ at time $0$. Let $\tau\in\mathcal T^{Y'}$ be a bounded stopping time. Then, applying (a generalised version of) Ito's formula and taking expectations at the  stopping time $\tau$, we get 
\begin{eqnarray*}
U(\pi) &=& \E \lt[ U \lt( \Pi^\prime(\tau) \rt)  \rt] 
-\E \Bigg[  \int_{0}^{\tau} \bigg(  \lambda^\prime(s) \lt( 1-  \Pi^\prime(s) \rt)  \pt_{\pi} U( \Pi^\prime(s)) \\
&& \hspace{25mm}+ \frac{b^2}{2\sigma^2} (\Pi^\prime)^{2}(s) \lt( 1- \Pi^\prime(s)\rt)^{2}  \pt^{2}_{\pi} U \lt( \Pi^\prime(s) \rt) \bigg)  \ud s\Bigg] \\
&\leq& \E \lt[ U(\Pi^\prime(\tau)) \rt]  -\E \Bigg[  \int_{0}^{\tau} \bigg( \lambda \lt( 1-  \Pi^\prime(s) \rt)  \pt_{\pi} U( \Pi^\prime(s))\\ 
&&\hspace{25mm}+ \frac{b^2}{2\sigma^2} (\Pi^\prime)^{2}(s) \lt( 1- \Pi^\prime (s)\rt)^{2}  \pt^{2}_{\pi} U \lt( \Pi^\prime(s) \rt)  \bigg) \ud s\Bigg] \\
&\leq& \E \lt[ U( \Pi^\prime(\tau)) \rt] +  \E \lt[ c\int_{0}^{\tau} \Pi^\prime(s) \ud s \rt] \\
&\leq& \E \lt[ 1 - \Pi^\prime(\tau) \rt] +  \E \lt[c \int_{0}^{\tau} \Pi^\prime(s) \ud s \rt],
\end{eqnarray*}
where we used the monotonicity of $U$ and the fact that 
\begin{eqnarray}\label{E:varin}
\lambda \lt( 1-  \pi \rt)  \pt_{\pi} U( \pi) 
+ \frac{b^2}{2\sigma^2}\pi^{2} \lt( 1- \pi\rt)^{2}  \pt^{2}_{\pi} U \lt(\pi\rt)+ c\pi\geq 0
\end{eqnarray}
at all points away from the optimal stopping boundary of Shiryaev's classical problem, compare \eqref{E:vareqn}.
Taking the infimum over bounded stopping times $\tau$, we get
$U(\pi) \leq V(0,\pi)$, which finishes the proof.
\end{proof}

\begin{remark}
\item
\begin{enumerate}
\item
The monotonicity in intensity does not easily extend to cases with unknown post-change drift by the same argument. In fact, one can check that in higher dimensions the partial derivatives $\frac{\partial V}{\partial \pi_i}$ are not necessarily all negative, which implies difficulties with extending the above proof to a more general setting. However, in the robustness result in Theorem~\ref{T:robustness} below we provide a partial extension in which models with general support for the drift magnitude and general intensities are compared with a fixed parameter model.
\item Though the authors expect the inequality in Theorem \ref{T:monotone} to hold also when one time-dependent intensity dominates another, the comparison with the constant intensity case was chosen to avoid additional mathematical complications that need to be resolved in order to apply Ito's formula to the value function of a time-dependent disorder detection problem.     
\end{enumerate}
\end{remark}

\subsection{Robustness}

Robustness concerns how a possible misspecification of the model parameters affects the performance of the detection strategy when evaluated under the real physical measure. In this section, we use coupling arguments to
study robustness properties with respect to the disorder magnitude and disorder time.
For simplicity, we assume that the parameters $\lambda$, $c$ and $\sigma$ are constant so that we have a time-independent case;
generalizations to the time-dependent case are straightforward but notationally more involved.

Thus we assume that the signal process follows
\begin{eqnarray}
X_{t}= B^{0}\Ind_{\{\Theta =0 \}}+ B^{1} \Ind_{\{ 0< \Theta \leq t \}} ,
\end{eqnarray}
where $B^{0}, B^{1}$ are random variables with distributions $\mu^{0}, \mu^{1}$ respectively, and $\Theta$ has the distribution $\nu_{\tilde{\pi}} := \tilde{\pi} \delta_{0}+(1-\tilde{\pi})\nu$, where $\nu$ is an exponential distribution with intensity $\lambda$. 
Let us simply write $\mu := (\mu^{0}, \mu^{1})$.

For a given $l\in\R\setminus\{ 0\}$, let $\Theta_l$ satisfy $\Theta_l\geq \Theta$ with distribution $\tilde{\pi} \delta_{0}+(1-\tilde{\pi})\nu_l$,
where $\nu_l$ is an exponential distribution with intensity $\lambda_l\leq \lambda$.
Let 
\begin{eqnarray} \nonumber
g_{l}(t,\tilde{\pi}, Y)
 &:=&  \frac{\tilde{\pi} e^{ \frac{l}{\sigma^{2}} Y_{t} -  \frac{ l^{2}}{2 \sigma^{2}} t}+ 
 (1- \tilde{\pi})   \lambda_l \int_{0}^t e^{  \frac{l}{\sigma^{2}} (Y_{t}-Y_\theta) -  \frac{ l^{2}}{2 \sigma^{2}} (t-\theta)}e^{-\theta/\lambda_l} \vd \theta}{ \tilde{\pi} e^{ \frac{l}{\sigma^{2}}  Y_{t} -  \frac{ l^{2}}{2 \sigma^{2}} t} + (1-\tilde{\pi}) \left( \lambda_l\int_{0}^t  e^{ \frac{l}{\sigma^{2}}  (Y_{t}-Y_\theta) -  \frac{ l^{2}}{2 \sigma^{2}}(t-\theta) }e^{-\theta/\lambda_l}  \vd \theta+  1-e^{-t/\lambda_l}  \right)},
 \label{E:PiiA}
\end{eqnarray}
compare \eqref{E:KalStr}. Also, we
introduce the notation 
\[Y^{\mu}_{t} := \int_0^t X_{u}\ud u+\sigma  W_{t},\]
\[Y^{\delta_{l}}_{t} :=l(t-\Theta_l)^{+}+  \sigma  W_{t},\] 
\[\tilde{\Pi}^{\delta_{l}}_{\delta_{l}}(t) := g_{l}(t, \tilde{\pi}, Y^{\delta_{l}})\] 
and
\[
\tilde{\Pi}^{\mu}_{\delta_l}(t) := g_{l}(t, \tilde{\pi}, Y^{\mu}).
\] 
Here $Y^{\mu}$ is the observation process for a setting in which the post-change drift has distribution $\mu$ and the disorder happens at $\Theta$. The process $Y^{\delta_l}$ is the observation process 
and $\tilde{\Pi}^{\delta_{l}}_{\delta_{l}}$ is the corresponding conditional probability process
in the situation of a post-change drift $l$ that occurs at $\Theta_l$. Moreover, the process $\tilde{\Pi}^{\mu}_{\delta_l}$
represents the conditional probability process calculated {\em as if} the drift change is described by $(\delta_l,\Theta_l)$
in the scenario where the true drift-change is given by $(\mu,\Theta)$.

Now, let $a:=a_{l}$ denote the optimal stopping boundary for the classical Shiryaev one-dimensional problem 
in the model $(\delta_l,\Theta_l)$, and define
\[\tau^{\delta_l}_{\delta_l} := \inf \{ t \geq 0 \,:\, \tilde{\Pi}^{\delta_l}_{\delta_l}(t) \geq a\},\]
\[
\tau^{\mu}_{\delta_l} := \inf \{ t \geq 0 \,:\, \tilde{\Pi}^{\mu}_{\delta_{l}}(t) \geq a\},
\]
and
\[V^{\mu}_{\delta_l} := \E[ \Ind_{ \{\tau^{\mu}_{\delta_l} < \Theta \}} + c( \tau^{\mu}_{\delta_l}-\Theta)^{+}].\]
Here $\tau_{\delta_l}^{\delta_l}$ is the optimal stopping time in the model $(\delta_l,\Theta_l)$, and
$\tau^{\mu}_{\delta_l}$ is the (sub-optimal) stopping time and $V^{\mu}_{\delta_l}$ is the corresponding cost
for someone who believes in $(\delta_l,\Theta_l)$, whereas the true model is $(\mu, \Theta)$.

Finally, let 
\[\tilde\Pi^\mu_t:=\P(\Ind_{\R\setminus\{0\}}(X_t)\vert \mathcal F^{Y^\mu}_t)=\Pi^{(1)}_t+...+\Pi^{(n)}_t\]
as in Section~\ref{S:2}, and define
\[\gamma^\mu_{\delta_l}:=\inf\{t\geq 0:\tilde\Pi^\mu_t\geq a\}.\]

\begin{theorem}[Robustness with respect to disorder magnitude and intensity]\label{T:robustness}
\item
\begin{enumerate}
\item
Suppose that $\inf (\supp \mu)> 0$ or $\sup (\supp \mu)< 0$, and let $l:= \argmin \limits_{x \in \supp(\mu)} |x|$. 
\begin{enumerate}
\item
Then
\begin{eqnarray}\label{E:rob1}
V^{\mu} \leq V^{\mu}_{\delta_l} \leq V^{\delta_{l}} + c\frac{\lambda-\lambda_l}{\lambda\lambda_l}(1-\tilde\pi)\,,
\end{eqnarray}
where $V^\mu$ and $V^{\delta_l}$ denote the minimal associated Bayes' risks for the models $(\mu,\Theta)$ and $(\delta_l,\Theta)$, respectively.
\item
Also,
\begin{eqnarray}\label{E:rob2}
V^{\mu} \leq \P(\Theta>\gamma_{\delta_l}^\mu)+c\E[(\gamma^\mu_{\delta_l}-\Theta)^+] \leq V^{\delta_{l}}.
\end{eqnarray}
\end{enumerate}
\item \label{Tp:Vr-rob}
Suppose $r:= \argmax \limits_{x \in \supp(\mu)} |x|$, and define $V^{\mu}_{\delta_ r}$ like $V^{\mu}_{\delta_ l}$ for $l=r$.
If $\lambda_r\geq \lambda$, then 
\begin{eqnarray}\label{E:rob3}
V^{\delta_{r}} \leq V^{\mu} \leq V^{\mu}_{\delta_ r}.
\end{eqnarray}
\end{enumerate}
\end{theorem}

\begin{remark}
Note that \eqref{E:rob1} and \eqref{E:rob3} correspond to situations in which the tester uses a misspecified model. More precisely, filtering and stopping are performed as if the underlying model had a one-point distribution as the disorder magnitude prior (the classical Shiryaev model). Such a situation may appear due to model miscalibration but is also relevant in situations with limited computational resources as the tester can deliberately choose to under/overestimate the actual parameters in order to use a simpler detection strategy. Equation \eqref{E:rob1} thus gives an upper bound for the expected loss when the classical Shiryaev model is employed. In \eqref{E:rob2}, on the other hand, filtering is performed according to the correct model but the simple Shiryaev threshold strategy (suboptimal) is used for stopping.
\end{remark}

\begin{proof}
\item
\begin{enumerate}
\item
\begin{enumerate}
\item
For definiteness, we consider the case $\inf (\supp \mu)> 0$ so that $l>0$; the other case is completely analogous.
First note that the suboptimality of $\tau^{\mu}_{\delta_l}$ yields $V^{\mu} \leq V^{\mu}_{\delta_l} $.
Next, observe that we have $Y^{\delta_l}_t= Y^\mu_t$ for all $0\leq t\leq \Theta$
and $Y^{\delta_l}_t\leq Y^\mu_t$ for all $t\geq 0$, and therefore 
\[\tilde{\Pi}^{\delta_{l}}_{\delta_{l}}(t) = \tilde{\Pi}^{\mu}_{\delta_l}(t) \quad \mbox{for } t\in[0,\Theta]\]
and
\[\tilde{\Pi}^{\delta_{l}}_{\delta_{l}}(t) \leq \tilde{\Pi}^{\mu}_{\delta_l}(t)\quad \mbox{for all } t\geq 0\]
by the filtering equation \eqref{E:PiiA}. Consequently, 
\[\tau^{\delta_l}_{\delta_l} \geq \tau^{\mu}_{\delta_l},\]
so
\begin{eqnarray}\label{E:onepart}
\E[(\tau^{\delta_l}_{\delta_l}-\Theta_l)^+]&\geq& \E[(\tau^{\mu}_{\delta_l}-\Theta)^+] -\E[(\Theta_l-\Theta)^+]\\
\notag
&=& \E[(\tau^{\mu}_{\delta_l}-\Theta)^+] -\frac{\lambda-\lambda_l}{\lambda\lambda_l}(1-\tilde\pi).
\end{eqnarray}
Moreover, since $\tilde{\Pi}^{\delta_{l}}_{\delta_{l}}(t)= \tilde{\Pi}^{\mu}_{\delta_l}(t)$ on the time interval $[0,\Theta]$, we have
\[\P(\tau^{\delta_l}_{\delta_l}<\Theta_l)\geq \P(\tau^{\delta_l}_{\delta_l}<\Theta)
=\P(\tau^{\mu}_{\delta_l}<\Theta),\]
which together with \eqref{E:onepart} yields
\begin{eqnarray*}
	V^{\delta_{l}} &=& \E[ \Ind_{ \{\tau^{\delta_l}_{\delta_l} < \Theta_l \}} + c( \tau^{\delta_l}_{\delta_l}-\Theta_l)^{+}] \\
	&\geq& \E[ \Ind_{ \{\tau^{\mu}_{\delta_l} < \Theta \}} + c( \tau^{\mu}_{\delta_l}-\Theta)^{+}] -c\frac{\lambda-\lambda_l}{\lambda\lambda_l}(1-\tilde\pi)\\
&=& V^{\mu}_{\delta_l} -c\frac{\lambda-\lambda_l}{\lambda\lambda_l}(1-\tilde\pi). 
\end{eqnarray*}
\item
The first inequality is immediate by suboptimality of $\gamma_{\delta_l}^\mu$. For the second one, let
$U$ be the value function of the classical Shiryaev problem so that $U(\tilde{\pi}) = V^{\delta_{l}}$. 
Then $U$ is $C^{2}$ on $[0, a_{l})\cup (a_{l},1]$ and $C^{1}$ on $[0,1]$, so applying It\^{o}'s formula to $U(\tilde{\Pi}_{t})$ and taking expectations at the bounded stopping time $\gamma_{\delta_l}^\mu\wedge k$, we get 
\begin{eqnarray*}
U(\tilde{\pi}) &=& \E[U(\tilde{\Pi}_{\gamma_{\delta_l}^\mu \wedge k})] - \E \lt[ \int_{0}^{\gamma_{\delta_l}^\mu \wedge k} \lambda (1 - \tilde{\Pi}_{u}) U'(\tilde{\Pi}_{u}) +  \frac{\hat X^{2}_{u}}{2\sigma^{2}} (1 - \tilde{\Pi}_{u})^{2} U''(\tilde{\Pi}_{u}) \ud u \rt]  \\
&\geq& \E[U(\tilde{\Pi}_{\gamma_{\delta_l}^\mu \wedge k})] - \E \lt[ \int_{0}^{\gamma_{\delta_l}^\mu \wedge k} \lambda_l (1 - \tilde{\Pi}_{u}) U'(\tilde{\Pi}_{u}) +  \frac{l^{2}}{2\sigma^{2}} \tilde{\Pi}_{u}^2(1 - \tilde{\Pi}_{u})^{2} U''(\tilde{\Pi}_{u}) \ud u \rt] \nonumber \\
&=& \E \lt[ U( \tilde{\Pi}_{\gamma_{\delta_l}^\mu \wedge k}) \rt] + \E \lt[ c \int_{0}^{\gamma_{\delta_l}^\mu \wedge k} \tilde{\Pi}_{u} \ud u \rt], 
\end{eqnarray*}
where monotonicity and concavity of $U$ were used in the inequality. Letting $k\to\infty$ gives 
\[U(\tilde{\pi})\geq \E \lt[  1-\tilde{\Pi}_{\gamma_{\delta_l}^\mu} \rt] + \E \lt[ c \int_{0}^{\gamma_{\delta_l}^\mu } \tilde{\Pi}_{u} \ud u \rt],\]
which finishes the proof of the claim.
\end{enumerate}

\item
Recall that
\[
\vd \tilde{\Pi}_t = \lambda (1 - \tilde{\Pi}_t) \ud t + \frac{\hat{X}_t}{\sigma} (1 - \tilde{\Pi}_t) \ud \hat W_t.
\]
Let $U(\tilde{\pi}) = V^{\delta_{r}}(\tilde{\pi})$.
Since $U$ is $C^{1}$ on $[0,1]$ and $C^{2}$ on $[0, a)\cup (a,1]$, where $a=a_r$ is the boundary in Shiryaev's problem with drift $r$ and intensity $\lambda_r$, applying It\^{o}'s formula to $U(\tilde{\Pi}_{t})$ and taking expectations at a bounded stopping time $\tau$ yields 
\begin{eqnarray}
U(\tilde{\pi}) &=& \E[U(\tilde{\Pi}_\tau)] - \E \lt[ \int_{0}^{\tau} \lambda (1 - \tilde{\Pi}_u) U'(\tilde{\Pi}_u) +  \frac{\hat X^{2}_u}{2\sigma^{2}} (1 - \tilde{\Pi}_u)^{2} U''(\tilde{\Pi}_u) \ud u \rt] \nonumber \\
&\leq& \E[U(\tilde{\Pi}_{\tau})] - \E \lt[ \int_{0}^{\tau} \lambda_r (1 - \tilde{\Pi}_u) U'(\tilde{\Pi}_u) +  \frac{r^{2}}{2\sigma^{2}} \tilde{\Pi}_{u}(1 - \tilde{\Pi}_u)^{2} U''(\tilde{\Pi}_u) \ud u \rt] \nonumber \\
&\leq& \E \lt[  U ( \tilde{\Pi}_\tau) \rt] + \E \lt[ c \int_{0}^{\tau} \tilde{\Pi}_u \ud u \rt] \label{E:FVI} \\
&\leq& \E \lt[ 1- \tilde{\Pi}_\tau \rt] + \E \lt[ c \int_{0}^{\tau} \tilde{\Pi}_u \ud u \rt] \label{E:Uineq} .
\end{eqnarray}
Here concavity was used for the first inequality, \eqref{E:FVI} follows from the fact that
\[
\lambda_r (1 - \tilde{\pi}) U'(\tilde{\pi}) + \frac{r^{2}}{2 \sigma^{2}} \tilde{\pi}(1 - \tilde{\pi})^{2} U''(\tilde{\pi}) + c \tilde{\pi} \geq 0, \quad \tilde{\pi} \in [0, a)\cup (a,1],
\]
and the inequality \eqref{E:Uineq} because $U(\tilde{\pi}) \leq 1 - \tilde{\pi}$.
Hence, since the same value $V^\mu$ is obtained if one in \eqref{E:OOS} restricts the infimum to only bounded stopping times,
\[ V^{\delta_r}=U \leq  V^{\mu}. \]
Lastly, since $\tau^{\mu}_{l}$ is a suboptimal strategy, we also have
\[
V^{\mu}  \leq V^{\mu}_{\delta_r} ,
\]
which finishes the claim. 
\end{enumerate}
\end{proof}

\begin{corollary}
\label{cor}
In the notation above, assume that $\lambda=\lambda_l$ so that there is no mis-specification of the intensity.
Moreover, assume that $\supp(\mu)\subseteq [l,r]$, where $0<l<r$. Then 
\[V^{\delta_r}\leq V^\mu\leq V^{\delta_l},\]
so monotonicity in the disorder magnitude holds when comparing with deterministic magnitudes.
Furthermore, 
\[0\leq V^{\mu}_{\delta_l}-V^\mu\leq V^{\delta_l}-V^{\delta_r},\]
so the increase in the Bayes' risk due to underestimation (with a constant) of the disorder magnitude is bounded by the difference of two value functions of the classical Shiryaev problem.
\end{corollary}

We finish with some implications concerning the stopping strategy $\tau_{\mathcal{D}}:= \inf\{ t\geq0\,:\, \Pi_t \in \mathcal{D}\}$, where $\mathcal{D} = \{ \pi \in \Delta^n\,:\, V(\pi) = 1-\pi\}$ is a standard abstractly defined optimal stopping set, see \cite{PS06}
(we now assume that we are in the case of time-independent coefficients so that the value function is merely a function of $\pi\in D^n$).
The concavity of $V$, compare Remark~\ref{T:concavity}, yields the existence of a boundary 
$\gamma \subset \Delta^n$ separating $\mathcal{D}$ from its complement $\Delta^n\setminus \mathcal{D}$. The following result provides a more accurate location of the boundary $\gamma$.

\begin{corollary}[Confined stopping boundary] \label{T:Squeezed}
Assume that the coefficients $c$, $\sigma$ and $\lambda$ are constant and that $\supp(\mu)\subseteq [l,r]$, where $0<l<r$. Let
$a_l$ and $a_r$ denote the boundaries in the classical Shiryaev problem with disorder magnitude $l$ and $r$, respectively. Then 
\begin{eqnarray*}
a_{l} \leq  \inf \{ \|\pi \|_1\,:\, \pi \in \gamma \}\leq \sup \{ \|\pi \|_1\,:\, \pi \in \gamma \} \leq a_{r},
\end{eqnarray*}
i.e. the stopping boundary is contained in a strip.
Moreover, the optimal strategy $\tau_{\mathcal{D}}$ satisfies 
\[
1- a_{r} \leq \P( \tau_{\mathcal{D}}<\Theta \, | \, \F^{Y}_{\tau_{\mathcal{D}}}) \leq 1-a_{l}.
\]
\end{corollary}

\end{document}